\documentclass{amsart}
\usepackage{tikz,makecell}
\usepackage[mathscr]{eucal}
\usepackage{amssymb}

\newtheorem{theorem}{Theorem}[section]
\newtheorem{lemma}[theorem]{Lemma}
\newtheorem{cor}[theorem]{Corollary}

\newtheorem{defn}[theorem]{Definition}
\newtheorem{obs}[theorem]{Observation}

\newcommand{\deq}{\mathrel{\mathop:}=}
\newcommand{\alg}[1]{\mathbf{#1}}

\newcommand{\conv}[1]{#1\breve{\ }}
\newcommand{\comp}{\mathbin{;}}

\def\pw{\mathscr{P}}

\tikzstyle{every label}=[label distance=0pt]
\tikzstyle{bdot}[1.5]=[circle,fill,draw,minimum size=#1mm,inner sep=0pt]
\tikzstyle{dot}[1.5]=[circle,draw,minimum size=#1mm,inner sep=0pt]
\tikzstyle{every edge}=[draw=black,thick]

\title{Qualitative representations of chromatic algebras}

\author[Al Juaid]{Badriah Al Juaid$^1$}\email{b.al-juaid@latrobe.edu.au}
\author[Jackson]{Marcel Jackson$^2$}\email{m.g.jackson@latrobe.edu.au}
\author[Koussas]{James Koussas$^3$}\email{j.koussas@latrobe.edu.au}
\author[Kowalski]{Tomasz Kowalski$^4$}\email{tomasz.s.kowalski@uj.edu.pl}
\address{${}^{1,2,3}$ Department of Mathematical and Physical Sciences, La Trobe University, Melbourne, VIC 3086, Australia}
\address{${}^{4}$ Jagiellonian University, Department of Logic,
Grodzka 52, 31-044 Krakow, Poland}

\begin{document}

\maketitle
\begin{abstract}
Conventional Ramsey-theoretic investigations for edge-colourings of complete graphs are framed around avoidance of certain configurations.  Motivated by considerations arising in the field of Qualitative Reasoning, we explore edge colourings that in addition to 
forbidding certain triangle configurations also require others to be present.  These conditions have natural combinatorial interest in their own right, but also correspond to \emph{qualitative representability} of certain \emph{nonassociative relation algebras}, which
we will call \emph{chromatic}.
\end{abstract}

\section{Introduction}\label{sec:intro}
In an edge $n$-colouring of a complete graph, each triangle of edges consists of either one colour, two colours or three colours: monochromatic, dichromatic or trichromatic.  
In this article we explore edge-colourings determined by disallowed triangle colour combinations, but also requiring others.  Thus, disallowing monochromatic triangles restricts to edge-coloured complete graphs within the Ramsey bound $R(3,3,\dots,3)$.  But what if in addition to disallowing monochromatic triangles, we also impose the dual constraint that all remaining colour combinations (trichromatic and dichromatic) are present: is it possible to find such a network?  These are natural combinatorial considerations in their own right, but there is an additional motivation by way of the algebraic foundations of qualitative reasoning, which finds wide application in AI settings around scheduling \cite{all}, navigation \cite{LRW,PSB} and geospatial positioning amongst others \cite{rcc8}.  The constraint language underlying typical qualitative reasoning systems determines a kind of non-associative relation algebra (in the sense of Maddux \cite{Mad82}), which is attracting considerable attention from a theoretical computer science perspective; see \cite{DLMSDVW,HJK19,InaEuz,lig,WHW} for example.  The inverse problem of deciding if a suitably defined non-associative algebra arises from a concrete constraint network is shown to be \texttt{NP}-complete in \cite{HJK19}.  The present work focusses on a natural family of combinatorially intriguing cases, that we find have nontrivial solution, and provide some novel extensions of classically understood connections between certain associative relation algebras and combinatorial geometries, such as in Lyndon \cite{LR61}.


More formally, let $C^+ = \{1',c_1,\dots,c_n\}$ be a finite set of \emph{colours}, let $U$ be any
set and let 
$\lambda\colon U\times U \to C^+$ be a surjective map
such that $\lambda(x,y) = \lambda(y,x)$ and
$\lambda(x,y) = 1'$ if{}f $x=y$. Then,
$\{\lambda^{-1}(1'),\lambda^{-1}(c_i), \dots, \lambda^{-1}(c_n)\}$ is a  
partition of $U\times U$ with every $\lambda^{-1}(c_i)$ symmetric and
$\lambda^{-1}(1')$ the identity relation. We call $\lambda$
an \emph{edge $n$-colouring} of the complete graph with the set of vertices $U$,
informally thinking of $1'$ as an invisible colour. The set
$C = C^+\setminus \{1'\}$ is the set of \emph{proper colours}.  We will consider natural
conditions on the colourings, forbidding certain triangles and requiring others.
To forbid the occurrence of
monochromatic triangles for example, we can impose the following condition 
$$
\forall a,b,c\in C\colon  |\{a,b,c\}| = 1 \Rightarrow \lambda^{-1}(a)\circ
\lambda^{-1}(b)\cap \lambda^{-1}(c) = \varnothing 
$$
where $\circ$ is the relational composition.  We will see shortly that from the point of view of the
algebras, strengthening the condition above to an equivalence
$$
\forall a,b,c\in C\colon  |\{a,b,c\}| = 1 \Leftrightarrow \lambda^{-1}(a)\circ
\lambda^{-1}(b)\cap \lambda^{-1}(c) = \varnothing 
$$
is more natural. Here we have all monochromatic triangles forbidden, and 
moreover all non-monochromatic triangles must occur.  As noted earlier, it is not
immediately clear whether such colourings exist: the Ramsey Theorem gives an
upper bound on the size of the graphs, the occurrence of all non-monochromatic
triangles gives a lower bound -- conceivably greater than the upper bound. 

In general, letting $F$ be any subset of $\{1,2,3\}$ we arrive at 8 natural
conditions of the form  
$$
\forall a,b,c\in C\colon  |\{a,b,c\}| \in F \Leftrightarrow \lambda^{-1}(a)\circ
\lambda^{-1}(b)\cap \lambda^{-1}(c) = \varnothing 
$$
and 8 corresponding existence questions; these are tied precisely to the
algebraic properties of interest in Lemma \ref{lem:combinatorial} below.  We
will also consider the weaker 
conditions, with implication, but we will see that they do not determine a
unique algebraic object.

\section{Nonassociative relation algebras and their representations}
In this section we recall basic constructions and ideas from the theory of relation algebras, which we then tie to the combinatorial conditions discussed informally at the start of the article.  This culminates in Lemma \ref{lem:combinatorial} of the next section, which shows that the basic hierarchy of strengths of representability (strong implies qualitative implies feeble) in the algebraic setting correspond precisely to three natural constraints around forbidden colour combinations and required colour combinations.  The reader wishing to skip the algebraic connections can survive most of the remainder of the article with only the notation presented in Definition \ref{defn:chromatic} and by using Lemma \ref{lem:combinatorial} as a Rosetta stone to relate algebraic terminology to combinatorial properties on edge-coloured complete graphs.

Nonassociative algebras were introduced and first studied in~\cite{Mad82}.
A \emph{nonassociative relation algebra} or simply a
\emph{nonassociative algebra}  (NA)
is an algebra $\alg{A} = (A, \wedge, \vee, \comp, \conv{}, \neg, 0, 1, 1')$ with the following properties.
\begin{enumerate}
\item $(A, \vee, \wedge, \neg, 0, 1) $ is a Boolean algebra.
\item $(A, \comp, \conv{}, 1')$ is an involutive groupoid with unit.  That is, $\comp$ is a binary operation,~$\conv{}$ unary,  $1'$ nullary, and
  the following identities hold:
\begin{enumerate}
\item $1'\comp x = x = x\comp 1'$
\item $\conv{\conv{x}} = x$.
\end{enumerate}
\item The non-Boolean operations are \emph{operators}.  That is, the following further
identities hold:
\begin{enumerate}
\item $x\comp(y\vee z) = (x\comp y)\vee (x\comp z)$ and
$(x\vee y)\comp z = (x\comp z)\vee (y\comp z)$
\item $x\comp 0 = 0 = 0\comp x$
\item $\conv{(x\vee y)} = \conv{x}\vee\conv{y}$
\item $ \conv{0}= 0$.
\end{enumerate}
\item $(x\comp y)\wedge z=0$ iff  $(\conv{x}\comp z)\wedge y=0$ iff
$(z\comp \conv{y})\wedge x=0$.
\end{enumerate}
The equivalences in (4) are known as the \emph{triangle laws}, or \emph{Peircean
  laws}.
If the operation $\comp$  is associative, then
$\mathbf{A}$ is a \emph{relation algebra} (RA), in the sense of Tarski~\cite{tar}; see a monograph such as Hirsch and Hodkinson~\cite{HH02} or Maddux~\cite{Mad06}. For any set~$U$
and an equivalence relation $E\subseteq U\times U$, the 
powerset $\pw(E)$ carries a natural relation algebra structure.
The Boolean operations are set-theoretical,
$\comp$ is the relational
composition, $\conv{}$ the converse, and $1'$ the identity relation.
Denote this algebra by $\mathfrak{Re}_E(U)$. If $E = U\times U$ we write
$\mathfrak{Re}(U)$, so $\mathfrak{Re}(U)$ is the algebra of all binary relations
on $U$. A relation algebra $\mathbf{A}$
is \emph{representable} if there is an injective homomorphism 
$\phi\colon \mathbf{A}\to\mathfrak{Re}_E(U)$, for some
$U$ (called the \emph{base} of the representation) and~$E$.
In particular we have
\begin{enumerate}
\item[(r1)] $0^\phi=\varnothing$, $1^\phi= E$, $(1')^\phi = Id_U$,
\item[(r2)] $\phi$ is a homomorphism of Boolean algebras,
\item[(r3)] $(\conv{a})^\phi= \conv{(a^\phi)}$,
\item[(r4)] $(a\comp b)^\phi = a^\phi\circ b^\phi$.
\end{enumerate}
If $E = U\times U$ the representation is called \emph{square}.
To avoid confusion with `representation'
in a generic sense of `any kind of representation', we will refer to
representations above as \emph{strong representations}. 
Two weakenings of this notion of representation, which apply to nonassociative
algebras, were defined  in~\cite{HJK19} and we now recall them.  
A set $\mathcal{H}\subseteq \pw(U\times U)$ for some set $U$, is called a
\emph{herd} if  
\begin{enumerate}
\item[(H1)] $\mathcal{H}$ is a Boolean set algebra with top element $U\times U$,
\item[(H2)] $Id_U\in \mathcal{H}$,
\item[(H3)] $S\in\mathcal{H}$ implies $\conv{S}\in\mathcal{H}$. 
\end{enumerate}
In any herd $\mathcal{H}$, the smallest element $R$ containing $S \circ T$, if
it exists, is called 
the \emph{weak composition} of $S$ and $T$, and is often denoted by
$S\diamond T$.  When $\mathcal{H}$ is finite (or more generally, complete) as a Boolean algebra then $S \diamond T$ always exists as we may simply intersect the elements of $\mathcal{H}$ that fully contain the relation $S \circ T$.

A nonassociative algebra $\mathbf{A}$  
is said to be \emph{qualitatively representable} 
if there is a bijection $\phi$ from ${A}$ to a herd
$\mathcal{H}$ such that
\begin{enumerate}
\item[(q1)] $0^\phi=\varnothing$, $1^\phi=U\times U$, $(1')^\phi = Id_U$,
\item[(q2)] $\phi$ is a homomorphism of Boolean algebras,
\item[(q3)] $(\conv{a})^\phi= \conv{(a^\phi)}$,
\item[(q4)] $a^\phi\circ b^\phi\subseteq c^\phi
  \leftrightarrow a\comp b\leq c$.
\end{enumerate}
Note that (q4) states that $c^\phi$ is the smallest solution
to $c^\phi \supseteq a^\phi\circ b^\phi$, and so
$c^\phi = a^\phi\diamond b^\phi$.
If~(q4) is strengthened to (r4), then the qualitative representation
$\phi$ is a strong square representation.

On the other hand,
if~(q4) is weakened to $\varphi(a\comp b)\supseteq \varphi(a)\circ\varphi(b)$, we
obtain a still weaker notion of representation, called \emph{feeble}
in~\cite{HJK19}. The next lemma combines well known facts on representations of
relation algebras with parts of Lemmas~12 and~18 of~\cite{HJK19}.
\begin{lemma}\label{q-rep-char}
Let $\mathbf{A}$ be complete and atomic with the set of atoms $At(A)$, and
let $\varphi\colon\mathbf{A} \to \pw(U\times U)$ satisfy (1)--(3) above.
Then \textup{(R1)} is equivalent to \textup{(R2)},
\textup{(Q1)} is equivalent to \textup{(Q2)}, and \textup{(F1)} is
equivalent to \textup{(F2)}. 
\begin{itemize}
\item[(R1)] $\phi$ is a strong square representation of $\mathbf{A}$.
\item[(R2)] For all $a,b,c\in At(A)$ we have
$(a\comp b)\wedge c \neq 0$ iff for all $x,y\in U$ such that
$(x,y)\in c^\phi$, there exists $z\in U$ with
$(x,z)\in a^\phi$ and $(z,y)\in b^\phi$.
\item[(Q1)] $\phi$ is a qualitative representation of $\mathbf{A}$.
\item[(Q2)] For all $a,b,c\in At(A)$ we have
  $(a\comp b)\wedge c \neq 0$ iff there exist $x,y,z\in U$ such that
  $(x,z)\in a^\phi$, $(z,y)\in b^\phi$, $(x,y)\in c^\phi$.
\item[(F1)] $\phi$ is a feeble representation of $\mathbf{A}$.
\item[(F2)] For all $a,b,c\in At(A)$ if there exist $x,y,z\in U$ such that
$(x,z)\in a^\phi$, $(z,y)\in b^\phi$, $(x,y)\in c^\phi$, then  
$(a\comp b)\wedge c \neq 0$. 
\end{itemize}
\end{lemma}
Since $a,b,c$ are atoms, the equivalence
$(a\comp b)\wedge c \neq 0 \Leftrightarrow a\comp b\geq c$
holds in $\mathbf{A}$. Strong representations must reflect this
property by (R2), but qualitative and feeble representations do not need to.

\begin{defn}
Let $\mathbf{A}$ be a nonassociative algebra. The \emph{atom structure}
of $\mathbf{A}$ is the structure $At(\mathbf{A}) = (At(A),\breve{\ }, I,T)$,
where $At(A)$ is the set of atoms of $\mathbf{A}$,
$\breve{\ }$ is the converse operation of $\mathbf{A}$ restricted to atoms,
$I = \{x\in At(A) \mid  x\leq 1' \}$ is the set of
\emph{subidentity} atoms \textup(other atoms are called \emph{diversity} atoms\textup),
$T = \{(x,y,z)\in At(A)^3\mid  z \leq x\comp y\}$ is
the set of \emph{consistent} triples. Triples $(x,y,z)\notin T$ are called
\emph{forbidden}.

Conversely, for a relational structure $\mathcal{X} = (X, \conv{\ }, I, T)$
where $\conv{\ }$ is a unary function, $I\subseteq X$, and
$T\subseteq X^3$, the \emph{complex algebra} $\mathfrak{Cm}(\mathcal{X})$ of
$\mathcal{X}$ is the algebra
$(\pw(X),\cap,\cup,\comp, \neg, \conv{\ },\varnothing, X, I)$, where
for $R,S\subseteq X$, we put
$\conv{S} = \{\conv{s}: s\in S\}$, and
$S\comp R = \{u\in X: (s,r,u)\in T, \text{ for some } s\in S, r\in R\}$.  
\end{defn}

If $\mathbf{A}$ is an atomic nonassociative algebra, then
the map  $x\mapsto \{a\in At(A): a\leq x\}$
is an embedding of $\mathbf{A}$ into $\mathfrak{Cm}(At(\mathbf{A}))$. 
If $\mathbf{A}$ is moreover complete, then it is an isomorphism.
It is convenient and very common to present complete and atomic nonassociative
algebras in terms of their atom structures. 
The next lemma, proved in~\cite{Mad82}, characterises
structures that are 
atom structures of nonassociative algebras.

\begin{lemma}
Let $\mathcal{X} = (X, \conv{\ }, I, T)$ be a structure such that $I$ is a
subset of $X$, 
$T$ is a subset of $X^3$, and $\conv{\ }$ is a function satisfying
$\conv{\conv{a}} = a$. The following are equivalent\textup:
\begin{itemize}
\item $\mathcal{X}$ is the atom structure of some nonassociative algebra.    
\item For all $a,b,c\in X$ we have
\begin{itemize}
\item $b =c$ iff there is some $e\in I$ such that $(e,b,c)\in T$ 
\item if $(a,b,c)\in T$ then $(\conv{c},a,\conv{b})\in T$ and
$(\conv{b},\conv{a},\conv{c})\in T$. 
\end{itemize}   
\end{itemize}
\end{lemma}

The triples $(a,b,c)$, $(\conv{a},c,b)$, $(c,\conv{b},a)$,
$(b,\conv{c},\conv{a})$, 
$(\conv{c},a,\conv{b})$, $(\conv{b},\conv{a},\conv{c})$
are called the \emph{Peircean transforms} of $(a,b,c)$.  See Figure \ref{fig:Peircean} for a pictorial explanation of Peircean triples.
The relation $T$ above is always closed under Peircean transforms. 
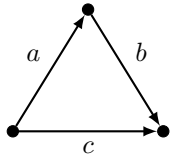
\begin{figure}[h]
\begin{tikzpicture}[auto]
\node [bdot] (1) at (0,0) {};
\node [bdot] (2)  at (1,1.62) {};
\node [bdot] (3) at (2,0) {};
\draw [thick,-latex] (1)  to node {$a$} (2);
\draw [thick,-latex] (2) to node {$b$} (3);
\draw [thick,-latex] (1) to node [swap] {$c$} (3);
\end{tikzpicture}
\caption{If $(a,b,c)\in T$, then the Peircean transforms correspond to all of the other cyclic traverses of this triangle.  If any one is present in a network, then all the others are too.  In this article, all algebras are symmetric, so the arrow direction is not important.}\label{fig:Peircean}
\end{figure}

There are many finite relation algebras that admit strong representations over infinite sets but no finite set, and in general deciding strong representability for finite relation algebras is algorithmically undecidable even \cite{HH01}.  Qualitative representability is decidable, though is \texttt{NP}-complete \cite[Theorem 15]{HJK19}.  Much of the simplification is due to the following useful fact. 
\begin{lemma}\label{lem:finiterep} \cite[Lemma 13]{HJK19}
If an atomic nonassociative algebra $\mathbf{A}$ admits a qualitative representation, then it admits a qualitative representation as relations on set of size at most $3|At(\mathbf{A})|$ points.
\end{lemma}

\section{Chromatic algebras}

\begin{defn}\label{defn:chromatic}
Let $S\subseteq \{1,2,3\}$.
We define a nonassociative algebra $\mathfrak{E}_{n+1}^S$
with $n+1$ atoms as the complex algebra $\mathfrak{Cm}(\mathcal{C})$ of
the structure $\mathcal{C} = (C^+, \conv{\ }, I, T)$, where
$C^+ = \{1',c_1,\dots,c_n\}$, $\conv{c} = c$ for any $c\in C^+$,
$I = \{1'\}$, and $T$ is given by 
\begin{enumerate}
\item $\forall b,c\in C^+ : b =c \Leftrightarrow (1',b,c)\in T$,
\item $\forall a,b,c\in C:  |\{a,b,c\}| \in S \Leftrightarrow
  (a, b, c)\in T$, where $C = C^+\setminus\{1'\}$. 
\end{enumerate} 
We call algebras $\mathfrak{E}_{n+1}^S$ \emph{chromatic}. 
\end{defn}
It is easy to see that $T$ is closed under the Peircean transforms.
Note that $S$ gives the types of \emph{consistent} triples rather than
forbidden ones, which would be more in line with the remarks in
Section~\ref{sec:intro}. We stated the definition this way to keep the notation
$\mathfrak{E}_{n+1}^S$ in agreement with~\cite{Mad06}. Expressing (2) in terms of the
operations in the complex algebra $\mathfrak{Cm}(\mathcal{C})$,
and identifying singletons with their single elements in a
set-theoretically incorrect but notationally convenient way
we get
$$
\forall a,b,c\in C:  |\{a,b,c\}| \in S \Leftrightarrow
  (a\comp b)\wedge c \neq 0.
  $$
Then, putting $F = \{1,2,3\}\setminus S$, we arrive at an equivalent version
of (2)
$$
\forall a,b,c\in C:  |\{a,b,c\}| \in F \Leftrightarrow
  (a\comp b)\wedge c = 0
$$
matching the remarks in Section~\ref{sec:intro} and the statement
of Lemma~\ref{q-rep-char}. With this the next lemma is not difficult to prove: the
representation (strong, qualitative or feeble) on the atoms is 
precisely $\lambda^{-1}$.  

\begin{lemma}\label{lem:combinatorial}
Let $S\subseteq \{1,2,3\}$ and $F = \{1,2,3\}\setminus S$. The following hold.
\begin{enumerate}  
\item 
$\mathfrak{E}_{n+1}^S$ is strongly representable if and only if
there exists an edge $n$-colouring $\lambda$ of a complete graph, satisfying
\begin{enumerate}
\item $\forall a,b,c\in C\colon  |\{a,b,c\}| \in F \Leftrightarrow \lambda^{-1}(a)\circ
  \lambda^{-1}(b)\cap \lambda^{-1}(c) = \varnothing$, and
\item $\forall a,b,c\in C\colon  |\{a,b,c\}| \in S \Leftrightarrow \lambda^{-1}(a)\circ
  \lambda^{-1}(b)\supseteq \lambda^{-1}(c)$.  
\end{enumerate}  
\item $\mathfrak{E}_{n+1}^S$ is qualitatively representable if and only if
there exists an edge $n$-colouring $\lambda$ of a complete graph, satisfying
\begin{enumerate}
\item $\forall a,b,c\in C\colon  |\{a,b,c\}| \in F \Leftrightarrow \lambda^{-1}(a)\circ
  \lambda^{-1}(b)\cap \lambda^{-1}(c) = \varnothing$.
\end{enumerate}
\item $\mathfrak{E}_{n+1}^S$ is feebly representable if and only if
there exists an edge $n$-colouring $\lambda$ of a complete graph
satisfying
\begin{enumerate}
\item[(a${}'$)] $\forall a,b,c\in C\colon  |\{a,b,c\}| \in F \Rightarrow \lambda^{-1}(a)\circ
  \lambda^{-1}(b)\cap \lambda^{-1}(c) = \varnothing$.
\end{enumerate}
\end{enumerate}
\end{lemma}  
Note that the left-hand sides of (a) and (b) above are negations of one another,
but the right-hand sides are not. Thus (a) alone forces $\comp$ to be represented
as weak composition, but (a) together with (b) force $\comp$ to be represented
as composition.

\begin{lemma}\label{lem:feeble-expansion}
If $S\subseteq S'$ and $\mathfrak{E}_{n+1}^S$ has a feeble
representation given as a colouring, the same colouring is also a feeble
representation of $\mathfrak{E}_{n+1}^{S'}$.
\end{lemma}
\begin{proof}
Let $F = \{1,2,3\}\setminus S$ and $F' = \{1,2,3\}\setminus S'$.
Then, $F'\subseteq F$ and the claim follows by
Lemma~\ref{lem:combinatorial}(3)(a${}'$).
\end{proof}

\section{Algebras $\mathfrak{E}_{n+1}^\varnothing$,
  $\mathfrak{E}_{n+1}^{\{1\}}$, $\mathfrak{E}_{n+1}^{\{1,2\}}$
and  $\mathfrak{E}_{n+1}^{\{1,2,3\}}$}\label{sec:easy}

For $\mathfrak{E}_{n+1}^\varnothing$ all triangles are
forbidden, so the only possible representation is the colouring of $K_2$ with a
single colour ($n=1$). For $\mathfrak{E}_{n+1}^{\{1\}}$ the only
possible representation is $K_m$ for $m\geq 3$ coloured with a single colour ($n=1$).
These representations are strong.

It follows from Theorem 422 of Maddux~\cite{Mad06} that
\(\mathfrak{E}_{n+1}^{\{1,2\}}\) is always strongly representable. 
Strong representations for $\mathfrak{E}_{n+1}^{\{1,2\}}$ can be easily built
using representability games of Hirsch and Hodkinson (see~\cite{HH02}), and  
a finite qualitative representation can be then extracted by Lemma \ref{lem:finiterep}.

Finite strong representations of $\mathfrak{E}_{n+1}^{\{1,2,3\}}$ are obtained
by Jipsen, Maddux and Tuza in~\cite{JMT95}.
A qualitative representation is much easier to
get: take the disjoint union of all possible triangles, and use an arbitrary
colour for all the missing edges.

\section{Algebras $\mathfrak{E}_{n+1}^{\{3\}}$}\label{sec:quasigr}

The algebra $\mathfrak{E}_{4}^{\{3\}}$ has a unique strong representation on $4$
points. For $n>3$, the algebras $\mathfrak{E}_{n+1}^{\{3\}}$ are not associative, so strong representability is impossible, making 
qualitative representability of particular interest.  Our construction for representability will be based around quasigroups, and we direct the 
reader to a text such as Smith \cite{Smi06} for further background.

Let $(Q,\cdot)$ be a commutative idempotent quasigroup on $\{0,\dots,n-1\}$.  We
define edge colourings $\lambda_1$ on the complete graph on $\{0,\dots,n-1\}$
and $\lambda_2$ on $\{-1,0,1,\dots,n-1\}$ into the colours\footnote{Note that
  the set of proper colours in this section is $\{c_0,\dots,c_{n-1}\}$,
  instead of $\{c_1,\dots,c_{n}\}$ used in the other sections. It makes
  calculations easier and should not cause confusion.}
 $\{1',c_0,\dots,c_{n-1}\}$ as follows:
\[
\lambda_1(i,j)=\begin{cases}
1'&\text{ (invisible) if }i=j\\
c_{i\cdot j}& \text{ otherwise}
\end{cases}
\]
and 
\[
\lambda_2(i,j)=\begin{cases}
1'&\text{ (invisible) if }i=j\\
\lambda_1(i,j) &\text{ if $i,j\neq -1$}\\
c_j&\text{ if $i=-1\neq j$}\\
c_i&\text{ if $j=-1\neq i$.}
\end{cases}
\]
Commutativity of $(Q,\cdot)$ ensures that $\lambda_i(j,k)=\lambda_i(k,j)$.
\begin{lemma}\label{lem:quasigrps}
If $(Q,\cdot)$ is a commutative idempotent quasigroup on $\{0,1,\dots,n-1\}$, then
$\lambda_1^{-1}$ and $\lambda_2^{-1}$ are feeble representations of $\mathfrak{E}_{n+1}^{\{3\}}$.  Moreover, the following properties are related by \textup{(1)}$\Leftrightarrow$\textup{(3)} and \textup{(1)}$\Rightarrow$\textup{(2)}\textup:
\begin{enumerate}
\item $\lambda_1$ is a qualitative representation of $\mathfrak{E}_{n+1}^{\{3\}}$,
\item $\lambda_2$ is a qualitative representation of $\mathfrak{E}_{n+1}^{\{3\}}$,
\item $(Q,\cdot)$ satisfies the following \emph{$3$-cycle condition}\textup: for each $x,y,z$ with $|\{x,y,z\}|=3$ there exists $u,v,w$ with $u\cdot v=x$, $v\cdot w=y$ and $w\cdot u=z$.
\end{enumerate}
\end{lemma}
\begin{proof}
It is clear that $\lambda_1$ and $\lambda_2$ are surjective, so to establish that both determine feeble representations we must ensure that there are no monochromatic or dichromatic triangles.  
For $\lambda_1$ this is just the cancellativity property of quasigroups: for $i,j,k\in\{0,\dots,n-1\}$ with $j\neq k$, we have $\lambda_1(i,j)\neq \lambda_1(i,k)$ because $i\cdot j=i\cdot k\Rightarrow j=k$, showing that no triangle contains two edges of the same colour.  
Thus~$\lambda_1^{-1}$ is a feeble representation.  Now observe that under $\lambda_1$, each point $i$ is incident to edges of all colours $c_0,\dots,c_{n-1}$ except colour $c_i$ (this is where idempotency of $(Q,\cdot)$ is used).  
The colouring $\lambda_2$ extends $\lambda_1$ to the extra vertex $-1$ by adding the missing colour for each vertex: an edge of colour $c_i$ between $-1$ and $i$, for each $i\in \{0,\dots,n-1\}$.  
It follows immediately that no non-trichromatic triangles are added under this extension: edges leaving $-1$ are coloured according to the name of the  vertex at the other end, so cannot be the same colour; and such an edge does not coincide in colour with any other edge incident to that vertex.  
Thus  $\lambda_2^{-1}$ also is  a feeble representations of~$\mathfrak{E}_{n+1}^{\{3\}}$.  

For the three conditions: (1) implies (2) because $\lambda_2$ extends $\lambda_1$, while according to the definition of $\lambda_1$, the $3$-cycle condition in (3) is precisely the condition that~$\lambda_1$ provides instances of each trichromatic triangle, so (1) is equivalent to (3).
\end{proof}
It is well known that commutative idempotent quasigroups exist for all and only odd orders, though we omit details because it emerges naturally from the proof of our main classification of qualitative representability $\mathfrak{E}_{n+1}^{\{3\}}$; Theorem \ref{thm:trichromatic} below.
We next recall a standard example of a commutative idempotent quasigroup, for every odd order, and verify that it satisfies the $3$-cycle condition.  

On universe $Q_n=\{0,\dots,n-1\}$ for an odd $n$, define a multiplication $\cdot$ by
$$
i\cdot j = \frac{i+j}{2}
$$
where $\frac{x}{2}$ stands for the unique integer $u\in \mathbb{Z}_n$
such that $2u = x \pmod n$, which exists because $n$ is odd. 

\begin{lemma}\label{lem:Qn}
The quasigroup $(Q_n,\cdot)$ satisfies the $3$-cycle condition.
\end{lemma}
\begin{proof}
Let $|\{i,j,k\}| = 3$ (equivalently $|\{c_{i},c_{j},c_{k}\}| = 3$).
We claim that it is possible to
choose  $a,y_{a},x_{a} \in Q_n$ , so that we have $i=a \cdot x_{a},$  $j=a
\cdot y_{a},$ $k=x_{a}\cdot y_{a}$, verifying the $3$-cycle condition.  By the unique solution property of
$(Q_n, \cdot)$, for each $a\in Q_{n}$, there 
exists $x_{a}, y_{a}\in Q_n$, where 
\begin{align*}
a \cdot x_{a} &= i\\
a \cdot y_{a} &= j.
\end{align*}
Then a simple calculation shows that 
$$
x_{a}\cdot y_{a}=i+j-a \pmod  n
$$
holds. It follows that choosing $a$ to be $k-i-j\pmod n$ provides the desired solution.  
\end{proof}
%
%

\begin{theorem}\label{thm:trichromatic}
The algebra $\mathfrak{E}_{n+1}^{\{3\}}$ is qualitatively representable if and only if
$n\geq 3$ is odd. Up to isomorphism, every representation of $\mathfrak{E}_{n+1}^{\{3\}}$ arises
from a representation of the form $\lambda_1^{-1}$ or $\lambda_2^{-1}$ of Lemma~\ref{lem:quasigrps}.
\end{theorem}
\begin{proof}
Assume $\mathfrak{E}_{n+1}^{\{3\}}$ is qualitatively representable for some $n$. Then,
there exists a finite representation, over a complete graph $K_m$.
Equivalently, there is a colouring map
$\lambda\colon K_m\times K_m\to C$, with $|C| = n$. 
Consider an arbitrarily chosen vertex $v\in K_m$. Since dichromatic and
monochromatic triangles are forbidden, all adjacent edges must be of different
colours, so $m-1\leq n$. As there are $n\choose 3$ different trichromatic
triangles to realise, we must have $m\geq n$. So, $n\leq m\leq n+1$.

If $m = n+1$, then every vertex has adjacent edges of all colours. 
Since edges of the same colour cannot be adjacent, the edges of any given colour
form a disjoint union of copies of $K_2$ and there are no isolated vertices. 
So the number of vertices is even, hence $n$ is odd. 

Next, assume $m = n$. Since there are ${n \choose 3}$ trichromatic triangles
to realise, each necessary triangle appears exactly once in the representation.
By symmetry of the colouring condition, each edge of a given colour appears the
same number of times, and as there are $\frac{n(n-1)}{2}$ edges and $n$ colours,
each edge appears $\frac{n-1}{2}$ times, so $n$ is odd.
Thus qualitative representations can exist only for odd orders, and Lemma~\ref{lem:Qn} and Lemma~\ref{lem:quasigrps} show that all odd orders of size at least $3$ are possible.

For the second part, first consider a representation
on $n$ vertices. Each vertex has adjacent edges of all but one colour,
call it the missing colour. Suppose there are vertices $v$ and $w$ such that
they miss the same colour, say $m$, so that no edge coloured $m$ is adjacent to
either $v$ or $w$. As each edge appears 
$\frac{n-1}{2}$ times, $K_n\setminus\{v,w\}$ contains $\frac{n-1}{2}$ edges
coloured $m$. These edges are disjoint, so $K_n\setminus\{v,w\}$ must contain
$n-1$ vertices, in contradiction to $|K_n\setminus\{v,w\}| = n-2$.
It follows that each vertex has a different missing colour. It
follows further that adding one vertex, say $v_{-1}$ to $K_n$
and letting $\lambda(v_{-1},v_i)$ be the missing colour of $v_i$,
we obtain a representation on $n+1$ vertices. So, every representation on
$n$ vertices is uniquely extendable to a representation on $n+1$ vertices.
Equivalently, every representation on $n$ vertices can be obtained from a
representation on $n+1$ vertices by removing one vertex.

Now, consider a representation on $n+1$ vertices. Numbering the vertices
$v_{-1}$, $v_0,\dots,v_{n-1}$ in such a way that $\lambda(v_{-1},v_i) = c_i$
for any $i\neq -1$, we define multiplication on the set of
indices $\{0,\dots,n-1\}$ putting $i\cdot i$ and $i\cdot j = k$ if
$i\neq j$ and $\lambda(v_i,v_j) = c_k$. It is then routine to verify that
$\{0,\dots,n-1\}$ is a commutative idempotent quasigroup, and that the colouring coincides with $\lambda_2$ (restricting to coincide with $\lambda_1$ amongst the vertices $v_0,\dots,v_{n-1}$).
\end{proof}

The two qualitative representations of $\mathfrak{E}_{n}^{\{3\}}$ can be
visualised on the complex plane as follows. Let $v_{-1}$ be the origin, and
$v_0,\dots,v_{n-1}$ be the $n$-th roots of unity. Put
$\lambda(v_{-1},v_{0}) = c_0 = \lambda(v_j,\overline{v_j})$,
for $j\in \{1,\dots, \lfloor \frac{n}{2} \rfloor\}$, where
$\overline{v_j}$ is the complex conjugate of $v_j$. Rotate and repeat with a
different colour, as in Figure~\ref{complex}. To get a representation
on $n$ points remove the origin and its outgoing edges.

\begin{figure}
\begin{tikzpicture}
\node[bdot,label=-45:$v_{-1}$] (orig) at (0,0) {};    
\node[bdot,label=-45:$v_0$] (v0) at (0:1.5) {};
\node[bdot,label=360/7:$v_1$] (v1) at (360/7:1.5) {};
\node[bdot,label=2*360/7:$v_2$] (v2) at (2*360/7:1.5) {};
\node[bdot,label=3*360/7:$v_3$] (v3) at (3*360/7:1.5) {};
\node[bdot,label=4*360/7:${v_4=\overline{v_3}}$] (v4) at (4*360/7:1.5) {};
\node[bdot,label=5*360/7:${v_5=\overline{v_2}}$] (v5) at (5*360/7:1.5) {};
\node[bdot,label=6*360/7:${v_6=\overline{v_1}}$] (v6) at (6*360/7:1.5) {};  
\draw[thick] (orig)--(v0);
\draw[thick] (v1)--(v6);
\draw[thick] (v2)--(v5);
\draw[thick] (v3)--(v4);
\draw[dashed] (orig)--(v1);
\draw[dashed] (v2)--(v0);
\draw[dashed] (v3)--(v6);
\draw[dashed] (v4)--(v5);
\draw[->] (-2,0)--(2,0);
\draw[->] (0,-2)--(0,2);
\end{tikzpicture}
\caption{First stages of drawing a qualitative representation of
  $\mathfrak{E}_{7}^{\{3\}}$ on the complex plane. }\label{complex}
\end{figure}
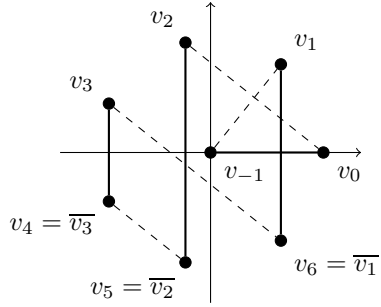

For feeble representations, we only need to consider $\mathfrak{E}_{n+1}^{\{3\}}$
for an even $n$. Taking a qualitative representation for
$\mathfrak{E}_{n}^{\{3\}}$, and re-colouring a single arbitrarily chosen edge with a new
colour gives a feeble representation of $\mathfrak{E}_{n+1}^{\{3\}}$. Hence, 
$\mathfrak{E}_{n+1}^{\{3\}}$ is feebly representable for all $n\geq 3$.
Now using Lemma~\ref{lem:feeble-expansion} we immediately obtain the following.

\begin{obs}\label{obs:3-feeble}
Let $3\in S$. Then $\mathfrak{E}_{n+1}^S$ is feebly representable for all $n\geq 3$.
\end{obs}

\section{Algebras $\mathfrak{E}_{n+1}^{\{2\}}$ } \label{sec:2}

The algebras $\mathfrak{E}_{n+1}^{\{2\}}$ are associative.
The pentagon algebra is the unique strong representation of
$\mathfrak{E}_{3}^{\{2\}}$, and it is easy to show that
strong representations do not exist for $n>2$.
At the other extreme, feeble representations always exist.

\begin{obs}\label{obs:2-feeble}
Let $2\in S$. Then $\mathfrak{E}_{n+1}^S$ is feebly representable for all $n\geq 2$.
\end{obs}

\begin{proof}
By Lemma~\ref{lem:feeble-expansion}, it suffices to show that $\mathfrak{E}_{n+1}^{\{2\}}$ is feebly representable for
all $n\geq 2$.  
Let $V = \{v_0,\dots, v_n\}$ be a set of vertices, and $C = \{c_1,\dots,c_n\}$ a
set of proper colours. Define a colouring
$\lambda\colon V^2\to C$ by $\lambda(v_i,v_j) = c_j$ if $i<j$. Then any triangle
that occurs is dichromatic, and all colours are used, so we have a feeble representation. 
\end{proof}  

Qualitative representations do not exist for $n>2$, but proving this
is considerably more difficult than for strong representations. We will work our way
towards a contradiction. Assume a qualitative representation over a base $B$
exists, and let $v\in B$. 
The \emph{chromatic degree} $\mathrm{deg}_\chi(v)$ of $v$ is the number of
colours of its adjacent edges. If a vertex $v$ has $\mathrm{deg}_\chi(v) = n$,
we call $v$ \emph{chromatically saturated}. To distinguish clearly between triangles
given as sets of vertices, and triangles given as sets of edges, we
will use parentheses $(v_i,v_j,v_k)$
for vertices and square brackets $[\ell,p,q]$ for edges/colours.

\begin{lemma}\label{lem:extend}
If $\mathfrak{E}_{n+1}^{\{2\}}$ has a qualitative representation over a base $B$, and
$v_0\in B$, then it also has a 
qualitative representation over a base $B'$, extending $B$, in which $v_0$ is
chromatically saturated.  
\end{lemma}  

\begin{proof}
Assume $\mathfrak{E}_{n+1}^{\{2\}}$ has a qualitative representation over $B$. Pick an arbitrary $v_0\in B$.
We can assume $B = \{v_0,\dots,v_m\}$.   
If $\mathrm{deg}_\chi(v) = n$, there is nothing to prove, so assume
$\mathrm{deg}_\chi(v) = k<n$. Let $d$ be a colour not adjacent to $v_0$. Add a vertex $v'$
and extend the labelling by
$$
\lambda(v',v_i) =
\begin{cases}
d & \text{ if }  i = 0\\
\lambda(v_0, v_i) & \text{ otherwise}
\end{cases}
$$
We will show that this labelling is consistent. 
As the only possible inconsistency involves $v'$, consider
$v_i$, $v_k$, $v'$ with $|\{v_i,v_k,v'\}| = 3$. If
$i\neq 0\neq k$, then we have $\lambda(v',v_i) = \lambda(v_0,v_i)$ and
$\lambda(v',v_k) = \lambda(v_0,v_k)$. Since the triangle $(v_0,v_i,v_k)$ is
consistent, so is $(v',v_i,v_k)$. Next, assume $i\neq 0 = k$, so we
consider the triangle $(v',v_0,v_i)$. Then, we
have $\lambda(v', v_0) = d$, and $\lambda(v', v_i) = \lambda(v_0, v_i)$
by construction. By assumption, $\lambda(v_0, v_i)\neq d$, so the
triangle $(v',v_0,v_i)$ is consistent. 

Repeating the extension procedure sufficiently many times, produces the desired
qualitative representation.
\end{proof}

By Lemma~\ref{lem:extend}, we can assume that $\mathfrak{E}_{n+1}^{\{2\}}$ has a
qualitative representation over 
some $B$ such that there is a chromatically saturated vertex $v\in B$. Every vertex $u\in B$ defines a partition of
$B\setminus\{u\}$ 
by adjacent colours, i.e., $z\sim_u w$ if and only if
$\lambda(u,z) = \lambda(u,w)$. Let $v_0$ be chromatically saturated.
Pick a single representative of each equivalence class of $\sim_{v_0}$
and let the representatives be $v_1,\dots, v_n$, for the colours
$1,\dots, n$, so that $\lambda(v_0,v_i) = i$. Let
$G$ be the induced labelled subgraph of $B$ on the vertices
$\{v_0, \dots,v_n\}$. Define a binary relation $\sqsubset$ on $\{v_0, \dots,v_n\}$
by putting $v_i\sqsubset v_j$ if $\lambda(v_i,v_j) = j$. Note that
$v_0 \sqsubset v_i$ for every $i > 0$. 

\begin{lemma}\label{lem:order}
The relation $\sqsubset$ is a strict linear order on $\{v_0,\dots,v_n\}$.
\end{lemma}  

\begin{proof}
Antireflexivity is immediate from the definition.   
To show transitivity, assume $v_p\sqsubset v_q$ and $v_q\sqsubset v_r$. By
definition of $\sqsubset$, only $p$ can be $0$, and in this case the claim
follows immediately, as $\lambda(v_0,v_r) = r$ by definition.

Assume $p\neq 0$. Then $\lambda(v_p,v_q) = q = \lambda(v_0,v_q)$ and
$\lambda(v_q,v_r) = r = \lambda(v_0,v_r)$. Now, consider $\lambda(p,r)$.
Since $\lambda(v_0,v_p) = p$, we must have $\lambda(v_p,v_r)\in \{q,r\}\cap\{p,r\}$,
and since $p\neq q$, we get $\lambda(v_p,v_r) = r$ as required. 

Linearity is also immediate. Take arbitrary $v_p$ and $v_q$. If $p=0$ or
$q = 0$ the claim holds trivially. Otherwise, we have
$\lambda(v_p,v_q)\in\{\lambda(v_0,v_p), \lambda(v_0,v_q)\} = \{p,q\}$. 
\end{proof}

Although $G$ is a feeble representation of $\mathfrak{E}_{n+1}^{\{2\}}$
(isomorphic to the one of Observation~\ref{obs:2-feeble}),
it is not a qualitative representation of $\mathfrak{E}_{n+1}^{\{2\}}$,
as no triangle $(v_i,v_i,v_j)$ with
$i\sqsubset j$ is realised in $G$. 
Hence, if $B$ is the base of a qualitative
representation  of
$\mathfrak{E}_{n+1}^{\{2\}}$, then $G \neq B$. 
Consider a vertex $u\notin G$.
As $|G| = n+1$, by the pigeonhole principle we must have
$\lambda(u,v_p) = \lambda(u,v_q)$ for some distinct vertices $v_p,v_q\in G$.  Assume that $\ell$ denotes the value $\lambda(u,v_p) = \lambda(u,v_q)$. 

Renumbering the colours, we can assume that $\sqsubset$ coincides with the
usual strict order on natural numbers. From now on we will work under this
assumption, using $<$ for the induced order on colours. We will refer to
$G$ as a \emph{naturally ordered subgraph} of $B$.  While there may be many subgraphs of this form, the proofs that follow do not require us to deviate from a single chosen one.  In contrast, will consider many different choices of $u\notin G$, and the values $p,q,\ell$ clearly depend on the choice of $u$.

We can assume
$p<q$, so  $\lambda(v_p,v_q) = q$, and thus, $\ell \neq q$. We can also assume
that $q$ is the largest vertex (amongst $v_0,\dots,v_n$) with a repeat, that is,
such that  $\lambda(u,v_s) = \lambda(u,v_q)$ holds for some $s\neq q$.
Thus, for every distinct $s,t\geq q$ we have $\lambda(u,v_s) \neq
\lambda(u,v_t)$. We will keep these assumptions fixed throughout the rest of the
section, and invite the reader to draw pictures while reading. 

\begin{lemma}\label{lem:about-ell}
Let $G$ be a naturally ordered subgraph of $B$, and let $u\notin G$.
The following hold\textup:
\begin{enumerate}
\item If $p\neq \ell$, then, for every $r \leq p$ we have $\lambda(u,v_r)=\ell$.
\item If $p = \ell$, then, for every $r < p$ we have $\lambda(u,v_r)=q$.
\item For every $r,r' < p$ we have $\lambda(u,v_r) = \lambda(u,v_{r'})$.
\item If $p<\ell<q$, we have $\lambda(u,v_\ell)=q$.
\item For every $r\neq \ell$ with $p<r<q$, we have $\lambda(u,v_r)=\ell$.
\item For every $r\neq \ell$ with $q<r$, we have $\lambda(u,v_r)=r$.
\end{enumerate}
\end{lemma}

\begin{proof}
For (1), let $v_r\sqsubset v_p$. Then $\lambda(v_r,v_p) = p$ and
$\lambda(v_r,v_q) = q$, so $\lambda(u,v_r)\in\{\ell,p\}\cap\{\ell,q\}$, as $p\neq\ell$.
Hence $\lambda(u,v_r) = \ell$. 

For (2), let $v_r\sqsubset v_p$. By transitivity, $v_r \sqsubset v_q$, so $\lambda(v_r, v_q) = q$, hence $\lambda(u,v_r) \in \{\ell,q\}$. As $p = \ell$, we  also have $\lambda(v_r,v_p) = \ell$, so $\lambda(u,v_r) \neq \ell$, hence $\lambda(u,v_r) = q$.

%

Next, (3) follows immediately from (1) and (2).

For (4), let $v_p\sqsubset v_\ell\sqsubset v_q$. Then 
$\lambda(u,v_\ell)\in \{\ell,q\}$, as $\lambda(u,v_q) = \ell$ and $\lambda(v_\ell,v_q) = q$. Since $\lambda(u,v_p) = \ell =
\lambda(v_\ell,v_p)$, we also have $\lambda(u,v_\ell)\neq \ell$, hence
$\lambda(u,v_r) = q$.

For (5), let $v_p\sqsubset v_r\sqsubset v_q$ and $r\neq \ell$. Then, 
$\lambda(v_p,v_r) = r$ and $\lambda(v_q,v_r) = q$, so 
$\lambda(u,v_r)\in\{\ell,r\}\cap\{\ell,q\}$, as $r\neq \ell$.
Hence $\lambda(u,v_r) = \ell$.

For (6), we have $\lambda(u,v_r)\in\{\ell,r\}$, but  $\lambda(u,v_r) = \ell$
contradicts the assumption that $v_q$ is the largest vertex with a repeat.
Hence, $\lambda(u,v_r) = r$.
\end{proof}

\begin{lemma}\label{lem:ell}
Let $G$ be a naturally ordered subgraph of $B$, and let $u\notin G$.  
Then, $v_\ell = v_{q-1}$ or $v_\ell = v_{q+1}$. Therefore,
$\ell = q-1$ or $\ell = q+1$.
\end{lemma}

\begin{proof}
We will apply Lemma~\ref{lem:about-ell} repeatedly, and referring to numbers (1)--(6) under the tacit understanding that  these will be the items in Lemma~\ref{lem:about-ell}.  Firstly, assume that \(p = \ell\) and \(p < r < q\). From (5), we get \(\lambda(u,v_r) = \ell\). By assumption, \(\lambda(u,v_q) = \ell\). By transitivity, if \(i < p\), then \(\lambda(v_i,v_r) = r\), and \(\lambda(v_i, v_q) = q\), hence \(\lambda(u,v_i) \in \{\ell, r\} \cap \{\ell, q\}\). Since \(r \neq q\), we have \(\lambda(u,v_i) = \ell\). This contradicts (2), so such a configuration cannot occur.  As \(p > 0\), such an \(i\) always exists. So, when \(p = \ell\), we require that no such \(r\) exists, i.e., that \(\ell = p = q-1\). We can now assume that \(p \neq \ell\).

Now assume $p < \ell < q-1$. Then \(p < q-1 <q\), so by (5), we have \(\lambda(u,v_{q-1}) = \ell\). As \(\lambda(v_{\ell},v_{q-1}) = q-1\) and \(\lambda(v_{\ell},v_{q}) = q\), we have $\lambda(u,v_\ell)\in\{\ell,q\}\cap\{\ell,q-1\} = \{\ell\}$. This contradicts (4), as \(\ell \neq q\). It follows that $p < \ell < q$ implies $\ell = q-1$.

Assume $q< \ell$. Clearly, \(\lambda(v_q, v_\ell) = \ell\). By (6), \(q+1 < \ell\) means \(\lambda(u,v_{q+1}) = q+1\), so \(\lambda(u,v_\ell) \in \{\ell, q+1\}\). But \(\lambda(u,v_q) = \ell = \lambda(v_q, v_\ell)\), so we have \(\lambda(u,v_{\ell}) = q+1\), which contradicts our minimality assumption on repeated labels of edges with \(u\). Thus, we must have \(\ell = q+1\) if \(q < \ell\).

Assume $0 < \ell < p$. Then by (1), we have $[\ell,\ell,\ell] = [\lambda(v_\ell,v_0), \lambda(u,v_\ell), \lambda(u,v_0)]$, which is a contradiction, as monochromatic triangles are forbidden.

Finally, the case $q = \ell$ is impossible due to $[\lambda(u,v_q), \lambda(u,v_p), \lambda(v_p,v_q)]$, so the results cover all cases.
\end{proof}

\begin{lemma}\label{lem:3-cases}
Let $G$ be a naturally ordered subgraph of $B$, and let $u\notin G$.
Let $p$, $q$ and $\ell$ be as in Lemmas~\ref{lem:about-ell} and~\ref{lem:ell}.  
Then exactly one of the following three cases  occurs\textup:
\begin{enumerate}
\item $\lambda(u,v_i) = q+1$ for all $i\in\{0,\dots,q\}$,\\
  $\lambda(u,v_{q+1})\in\{1,\dots,q\}$,\\
  $\lambda(u,v_j) = j$ for all $j\in\{q+2,\dots, n\}$.
\item $\lambda(u,v_i) = q-1$ for all $i\in\{0,\dots,q-2\}\cup\{q\}$,\\
$\lambda(u, v_{q-1}) = q$,\\
$\lambda(u,v_j) = j$ for all $j\in\{q+1,\dots, n\}$.   
\item $\lambda(u,v_i) = q$ for all $i\in\{0,\dots,q-2\}$,\\
$\lambda(u,v_{q-1}) = \lambda(u,v_{q}) = q-1$,\\
$\lambda(u,v_j) = j$ for all $j\in\{q+1,\dots, n\}$.      
\end{enumerate}  
\end{lemma}
\begin{proof}
By Lemma~\ref{lem:ell}, we
have either $\ell = q+1$ or $\ell = q-1$.
Thus, $\ell = q+1$ implies $\lambda(u,v_q) = q+1$, and
then cases~(1), (5), and~(6) of Lemma~\ref{lem:about-ell} apply, showing that
the first and the third items of (1) above hold.  For the second item,
note that $\lambda(u,v_{q+1}) \geq q+1$ contradicts the assumption of
$v_q$ being the largest vertex with a repeat. 

If $\ell = q-1$, then we have $\lambda(u,v_q) = q-1=\lambda(u, v_p)$.  Recalling that $p<q$, we will discover that case (3) above will result from $p=q-1$ and case (2) from $p<q-1$.

If $p = q-1$, then case~(2) of Lemma~\ref{lem:about-ell} applies
yielding $\lambda(u,v_j) = q$ for all $j\in\{0,\dots, q-2\}$, and these together
with case~(6) of the same lemma imply that~(3) above holds.

Finally, if $\ell = q-1$ and $\lambda(u, v_p) = q-1$ for $p<q-1$, then
$\lambda(u,v_q) = q-1$, and by case~(4) of Lemma~\ref{lem:about-ell} we have
$\lambda(u,v_{q-1}) = q$. 
As $p\neq q-1$, case~(1) of Lemma~\ref{lem:about-ell} applies, giving
$\lambda(u,v_j) = q-1$ for all $j\leq p$. Further, case~(5) yields
$\lambda(u,v_j) = q-1$ for all $j\in\{p+1,\dots,q-2\}$. Putting all of these together,
we have $\lambda(u,v_j) = q-1$ for all $j\in\{0,\dots, q-2\}\cup\{q\}$, and then
applying case~(6) we obtain $\lambda(u,v_j) = j$ for all $j>q$, as required by
(2) above. 
\end{proof}

\begin{cor}\label{1-forces-n}
Let $G$ be a naturally ordered subgraph of $B$, and let $u\notin G$.
Then, the following hold\textup:
\begin{enumerate}
\item If $\lambda(u,v_{i}) < i$, then $\lambda(u,v_j) = j$ for all $j>i$.
\item If for some $j$ we have $\lambda(u,v_j) = j$ and $\lambda(u,v_{j-1}) =
  j-1$, then there is an $i<j-1$ with $\lambda(u,v_i) < i$.
\item $|\lambda(u,v_0)-\lambda(u,v_1)|\leq 1$.  
\end{enumerate}
\end{cor}

\begin{proof}
Follows by inspecting the cases of Lemma~\ref{lem:3-cases}.
\end{proof}  

It will be useful to state separately one special case of
Lemma~\ref{lem:3-cases}, covering the case where the largest repeat occurs
somewhere at the last three vertices.

\begin{lemma}\label{lem:3-cases-n}
Let $G$ be a naturally ordered subgraph of $B$, and let $u\notin G$.
Assume $\lambda(u,v_n)\neq n$. The exactly one of the following three cases occurs\textup:
\begin{enumerate}
\item $\lambda(u,v_i) = n$ for all $i\in\{0,\dots,n-1\}$,\\
  $\lambda(u,v_n)\in\{1,\dots,n-1\}$.
\item $\lambda(u,v_i) = n-1$ for all $i\in\{0,\dots,n-2\}\cup\{n\}$,\\
$\lambda(u, v_{n-1}) = n$.
\item $\lambda(u,v_i) = n$ for all $i\in\{0,\dots,n-2\}$,\\
$\lambda(u,v_{n-1}) = \lambda(u,v_{n}) = n-1$.
\end{enumerate}  
\end{lemma}

Now, by definition, $G$ does not
contain the triangle $[1,1,n]$, so if $B$ is the base of a qualitative
representation of 
$\mathfrak{E}_{n+1}^{\{1,3\}}$, there must exist vertices $a,b,c\in B$ realising this
triangle. We can assume $\lambda(a,b) = n$ and 
$\lambda(a,c) = 1 = \lambda(b,c)$. We have several cases to consider.

\begin{lemma}\label{all-out}
Let $a,b,c \notin G$. Then, $G\cup\{a,b,c\}$ is an inconsistent configuration.
\end{lemma}  

\begin{proof}
Consider $\lambda(a,v_n)$ and $\lambda(b,v_n)$. They cannot be both equal to
$n$, so assume $\lambda(a,v_n) \neq n$.
Taking $u = a$, we see that the following
possibilities can occur.

We begin with case (2) of Lemma~\ref{lem:3-cases-n}. We have  
$\lambda(a,v_{i}) = n-1$ for all $i\in \{0,\dots,{n-2}\}\cup\{n\}$ and
$\lambda(a,v_{n-1}) = n$. This implies $\lambda(b,v_{n-1}) \neq n$.
Let $k = \lambda(b,v_{n-1})$.

(2.1) If $k < n-1$, then by Corollary~\ref{1-forces-n}(1) we get
$\lambda(b,v_n) = n$.
Hence, $\lambda(c,v_n) \in \{\lambda(c,a), \lambda(a,v_n)\}
\cap\{\lambda(c,b), \lambda(b,v_n)\} = \{1,n-1\}\cap\{1,n\} = \{1\}$.
So, $\lambda(c,v_n) = 1$ and then
Lemma~\ref{lem:3-cases-n}(1) applies, giving
$\lambda(c,v_i) = n$ for all $i<n$.
In particular, $\lambda(c,v_{n-2}) = n$, and since
$\lambda(c,a) = 1$ and $\lambda(a,v_{n-2}) = n-1$, we obtain
an inconsistent, trichromatic triangle $(a,c,v_{n-2})$. 

(2.2) If $k = n-1$, then $\lambda(c,v_{n-1})\in \{\lambda(c,a), \lambda(a,v_{n-1})\}
\cap\{\lambda(c,b), \lambda(b,v_{n-1})\} =
\{1,n\}\cap\{1,n-1\} = \{1\}$. So, $\lambda(c,v_{n-1}) = 1$ and 
thus by Corollary ~\ref{1-forces-n}(1) we have $\lambda(c,v_n) = n$. 
Then, however, we obtain
$\lambda(a,c) = 1$, $\lambda(a,v_n) = n-1$, and $\lambda(c,v_n) = n$; an
inconsistent, trichromatic triangle.

Next, cases (1) and (3) of Lemma~\ref{lem:3-cases-n} will be dealt with
together. For case (1) we have 
$\lambda(a,v_i) = n$ for all $i<n$. As $\lambda(a,b) = n$, it follows that
$\lambda(b,v_i)\neq n$ for all $i<n$. For case (3), we have 
$\lambda(a,v_{n-1}) = \lambda(a,v_{n}) = n-1$, and
$\lambda(a,v_{i}) = n$ for all $i<n-1$. Then,
$\lambda(b,v_{i}) \neq n$ for all $i<n-1$.
Therefore in both cases we have $\lambda(b,v_{i}) \neq n$ for all $i<n-1$.
Since $n>2$, we get $\lambda(b,v_0) \neq n\neq \lambda(b,v_1)$.

(1.1/3.1) If $\lambda(b,v_i)\neq 1$ for $i \in\{0,1\}$, then  
$\lambda(c,v_i) \in \{1,n\}\cap \{1,\lambda(b,v_i)\} = \{1\}$ for $i\in\{0,1\}$,
and then $\lambda(c,v_0) = \lambda(c,v_1) =
\lambda(v_0,v_1) = 1$; an inconsistent, monochromatic triangle. 

(1.2/3.2) Assume $\lambda(b,v_i) = 1$ for exactly one $i\in\{0,1\}$.
Putting $\{s,t\} = \{0,1\}$, let 
$\lambda(b,v_s) = 1$ and $\lambda(b,v_t) \neq 1$. Thus,
$\lambda(c,v_s) = n$ and $\lambda(c,v_t) = 1$, contradicting Corollary \ref{1-forces-n}(3).
\end{proof}  

\begin{lemma}\label{two-out-n}
Let $a,b\notin G$ and $c\in G$. Then, $G\cup\{a,b\}$ is an inconsistent configuration.
\end{lemma}  

\begin{proof}
By the assumption, $c = v_j$ for some $j$, so $\lambda(a,v_j) = 1 =
\lambda(b,v_j)$. As in the previous lemma, we can have at most one of 
$\lambda(a,v_n)$, $\lambda(b,v_n)$ equal to $n$, so without loss assume
$\lambda(a,v_n)\neq n$. We analyse possible
cases.

Case (1) of Lemma~\ref{lem:3-cases-n}. We have
$\lambda(a,v_i) = n$ for all $i \in \{0,\dots,n-1\}$, which can only happen
if $j = n$. But, $\lambda(b,v_j) = 1$ as well, so, as $1<n-1$, case (1) is the only
possibility for $\lambda(b,v_j)$, and thus we get
$\lambda(b,v_i) = n$ for all $i \in \{0,\dots,n-1\}$. In particular,
$\lambda(a,v_{n-1}) = \lambda(b,v_{n-1}) = \lambda(a,b) = n$, producing an inconsistent,
monochromatic triangle. 

Case (2) of Lemma~\ref{lem:3-cases-n}. We have 
$\lambda(a,v_{i}) = n-1$ for all $i\in \{0,\dots,{n-2}\}\cup\{n\}$ and
$\lambda(a,v_{n-1}) = n$. Since $\lambda(a,v_j) = 1$ for some $j$, 
it implies $n = 2$, contradicting the assumptions.

Case (3) of Lemma~\ref{lem:3-cases-n}. We have
$\lambda(a,v_{n-1}) = \lambda(a,v_{n}) = n-1$, and
$\lambda(a,v_{i}) = n$ for all $i\in\{0,\dots,n-2\}$.
Since $\lambda(a,v_j) = 1$ for some $j$, it implies $n = 2$, 
contradicting the assumptions.
\end{proof}

\begin{lemma}\label{two-out-1}
Let $b,c\notin G$ and $a\in G$. Then, $G\cup\{b,c\}$ is an inconsistent configuration.
\end{lemma}  

\begin{proof}
We have $a = v_j$ for some $j$. First we show that we can assume $j\neq n$.
If $j=n$, then $\lambda(c,v_n) = 1$, and
since $n>2$ only the case (1) of Lemma~\ref{lem:3-cases-n} applies.
It follows that $\lambda(c,v_i) = n$ for all $i<n$. Then
$\lambda(b,v_{n-1})\in\{\lambda(b,c),\lambda(c,v_{n-1})\} = \{1,n\}$.  
But $\lambda(b,v_n) = \lambda(v_n,v_{n-1}) = n$, so 
$\lambda(b,v_{n-1})\neq n$, hence $\lambda(b,v_{n-1}) = 1$.  
Thus we have a $[1,1,n]$ triangle on $(c,b,v_{n-1})$, so switching the roles of
$b$ and $c$, and putting $a = v_{n-1}$ we get
$\lambda(a,b) = n$ with $a\neq v_n$, and $\lambda(a,c) = 1 = \lambda(b,c)$.

Now we can assume $a = v_j$ for some $j\neq n$. Then,
since \(\lambda(b,c) = 1\) and \(\lambda(c,v_n) \in
\{\lambda(c,v_j),\lambda(v_j,v_n)\} = \{1,n\}\), we have \(\lambda(b,v_n) \in
\{1,n\} \). However, \(\lambda(b,v_j) = n = \lambda(v_j,v_n) \), so
\(\lambda(b,v_n) = 1\). 
This excludes cases (2) and (3) of
Lemma~\ref{lem:3-cases-n}, as $n > 2$. The remaining case (1)  
gives $\lambda(b,v_i) = n$
for all $i< n$ and then it follows that $\lambda(c,v_k)
\in\{\lambda(c,v_j),\lambda(v_j,v_k)\}\cap\{\lambda(c,b),\lambda(b,v_k)\} =
\{1,\lambda(v_j,v_k)\}\cap\{1,n\} = \{1\}$, for all
$k\in\{0,\dots,j-1\}\cup\{j+1,n-1\}$ as  $\lambda(v_j,v_k)\neq n$ for all such
$k$. Therefore,  $\lambda(c,v_k) = 1$ for all $k<n$.
In particular, $\lambda(v_0,v_1) = \lambda(c,v_1) = \lambda(c,v_0) = 1$,
producing an inconsistent, monochromatic triangle.
\end{proof}  

\begin{lemma}\label{one-out-1-1}
Let $c\notin G$ and $a,b\in G$. Then, $G\cup\{c\}$ is an inconsistent configuration.
\end{lemma}

\begin{proof}
Let $a = v_i$ for some $i<n$ and $b = v_n$. As
$\lambda(c,v_n) = 1$ and $n>2$, the only applicable case of
Lemma~\ref{lem:3-cases-n} is 
(1). However, it implies $\lambda(c,v_i) = n$ for all $i<n$, contradicting
$\lambda(c,a) = 1$. 
\end{proof}  

\begin{lemma}\label{one-out-1-n}
Let $b\notin G$ and $a,c\in G$. Then, $G\cup\{b\}$ is an inconsistent configuration.
\end{lemma}

\begin{proof}
We must have $\{a,c\} = \{v_0,v_1\}$. But then $\lambda(b,a)=n$ and $\lambda(b,c)=1$ contradicts Corollary \ref{1-forces-n}(3).
\end{proof}  

Lemma \ref{one-out-1-n} completed the last possibility for the triangle
$\{a,b,c\}$, and as each possibility led to inconsistent networks, the proof of the claimed impossibility of qualitative representability is complete.
\begin{theorem}
The algebra $\mathfrak{E}_{n+1}^{\{2\}}$ is not qualitatively representable for any $n>2$.
\end{theorem}

\section{Algebras $\mathfrak{E}_{n+1}^{\{2,3\}}$}
The so-called \emph{Ramsey algebras}  $\mathfrak{E}_{n+1}^{\{2,3\}}$ (sometimes also known as Monk algebras and as Comer relation algebras) have received particularly ardent attention in the literature, with the obvious close connection to  classical Ramsey-theoretic considerations of avoidance of monochromatic triangles, as discussed in the introduction; \cite{GG55} and \cite{KS68} for example.  Strong representability presents very tightly constrained condition (see Lemma \ref{lem:combinatorial}) and proves too be particularly challenging.
Ramsey algebras appear in both of the texts \cite{HH02} and \cite{Mad06}, with strong representability for $1\leq n\leq 5$  established by Comer \cite{Com83} and later $n=6,7$ by Maddux \cite{Mad11}.  
The limits of strong representability were subsequently  pushed upward via a series of efforts of Kowalski \cite{Kow15}, Alm and Manske \cite{AM15} and finally Alm \cite{Alm17}, which is the current state of knowledge: strong representability is achievable for all $n\leq 2000$ except possibly $n=8,13$.
The cases of $n=8,13$ and $n>2000$ remain tantalisingly unresolved, though the finite field constructions in the above articles have been computationally verified as not providing solutions for $n=8,13$  (see A263308 in the Online Encyclop{\ae}dia of Integer Sequences \cite{almOEIS}).  Qualitative representability presents an interesting intermediate condition, and in this section we show that qualitative representability is possible for all $n$.

We employ a variant of Walecki's construction (see, e.g.,~\cite{Als08})
originally used to partition $K_{2k}$ into $k$ Hamiltonian paths.
Let the vertices of $K_{2n}$ be distributed
evenly on a circle. Colour a zigzagging path
in one colour, as in Figure~\ref{walecki}.
Then rotate the picture one step, change the colour and repeat. Rotating
$n$ times and using $n$ different colours gives the desired colouring of $K_{2n}$.

\begin{figure}
\begin{tikzpicture}
\coordinate (a) at (90:5/2);
\coordinate (b) at (70:5/2);
\coordinate (c) at (50:5/2);
\coordinate (d) at (30:5/2);
\coordinate (e) at (330:5/2);
\coordinate (f) at (310:5/2);
\coordinate (g) at (290:5/2);
\coordinate (h) at (270:5/2);
\coordinate (i) at (250:5/2);
\coordinate (j) at (230:5/2);
\coordinate (k) at (210:5/2);
\coordinate (l) at (150:5/2);
 \coordinate (m) at (130:5/2);
\coordinate (n) at (110:5/2);
\draw[thick] (a) -- (b) -- (n) -- (c) -- (m) -- (d) -- (l) -- (145.184830524717:2.03734350956);
\draw[thick, dotted] (141.4144813757:1.78706434649)--(145.184830524717:2.03734350956);
\draw[thick] (h) -- (i) -- (g) -- (j) -- (f) -- (k) -- (e) -- (325.184830525:2.03734350956);
\draw[thick, dotted]  (321.414481376:1.78706434649)--(325.184830525:2.03734350956);
\draw[thick, fill=white] (a) circle (0.065);
\draw[thick, fill=white] (b) circle (0.065);
\draw[thick, fill=white] (c) circle (0.065);
\draw[thick, fill=white] (d) circle (0.065);
\draw[thick, fill=white] (e) circle (0.065);
\draw[thick, fill=white] (f) circle (0.065);
\draw[thick, fill=white] (g) circle (0.065);
\draw[thick, fill=white] (h) circle (0.065);
\draw[thick, fill=white] (i) circle (0.065);
\draw[thick, fill=white] (j) circle (0.065);
\draw[thick, fill=white] (k) circle (0.065);
\draw[thick, fill=white] (l) circle (0.065);
\draw[thick, fill=white] (m) circle (0.065);
\draw[thick, fill=white] (n) circle (0.065);
\draw[above] (a) node {\(u_1\)};
\draw[right] (b) node {\(u_2\)};
\draw[right] (c) node {\(u_3\)};
\draw[right] (d) node {\(u_4\)};
\draw[right] (e) node {\(u_{n-2}\)};
\draw[right] (f) node {\(u_{n-1}\)};
\draw[right] (g) node {\(u_n\)};
\draw[below] (h) node {\(u_{n+1}\)};
\draw[left] (i) node {\(u_{n+2}\)};
\draw[left] (j) node {\(u_{n+3}\)};
\draw[left] (k) node {\(u_{n+4}\)};
\draw[left] (l) node {\(u_{2n-2}\)};
\draw[left] (m) node {\(u_{2n-1}\)};
\draw[left] (n) node {\(u_{2n}\)};
\end{tikzpicture}
\caption{The Walecki construction.}\label{walecki}
\end{figure}

It is clear that no monochromatic triangles occur, so  
it remains to be shown that all non-monochromatic triangles are realised.
To avoid cluttering subscripts, we will use $1, \dots, n$ as the colours,
instead of $c_1, \dots, c_n$. 
For each $n \in \mathbb{N}$, define a map
\(s_n \colon \mathbb{Z} \to \{1,\dots, n\}\) by
\(x \mapsto R_n(x-1)+1\), where $R_n(k)$ is $k \pmod n$.
In particular, we have $s_n(s_{2n}(a) + s_{2n}(b)) =
s_n(a+b)$ for all $a,b\in \mathbb{Z}$ and $n\in\mathbb{N}$.
We will use this fact repeatedly below to omit $s_{2n}$ within parentheses.

It can be seen from the construction that the colour of $(u_{1},u_{1+s})$ is
$\lceil\frac{s+1}{2}\rceil$, for each $s \in \{1,2,\dots,2n-1\}$.
Rotating $i$ steps clockwise, we get that 
for each \(i \in \{1,\dots,2n\}\) and
$s \in \mathbb{Z}\setminus\{2mn: m\in\mathbb{Z}\}$,
the colour of $(u_{i}, u_{s_{2n}(i+s)})$ is 
$$
s_n\bigg(\bigg\lceil \frac{s+1}{2} \bigg\rceil + i - 1\bigg).
$$
To see this, note that the colour of $(u_{1}, u_{s_{2n}(1+s)})$ is 
$s_n(\lceil\frac{s+1}{2}\rceil)$ and then shift the colour $i-1$ steps forward using the
rotational symmetry and the fact that the colours cycle with period $n$.

\begin{lemma}\label{map-s}
The colouring defined above realises all non-monochromatic triangles.
\end{lemma}

\begin{proof}
Take \(i,j,k \in \{1,\dots,n\}\) with \(i < j\).
We will show that there is a triangle coloured $(i,j,k)$.
We consider two cases: $k\neq i$ and $k\neq j$. Assume first that
$k\neq i$. Put
\[
\ell \deq s_{2n}(i + j - k), \quad s \deq 2k - 2j + 1, \quad \textrm{and} \quad
t \deq 2k - 2i. 
\]
Note that $s,t\in \mathbb{Z}\setminus\{2mn: m\in\mathbb{Z}\}$.
Hence the colour of $(u_\ell, u_{s_{2n}(\ell + s)})$ is
\[
s_n\bigg(\bigg\lceil \frac{(2k-2j+1)+1}{2} \bigg\rceil +( i + j - k) - 1\bigg) =
s_n(k-j+1 + i + j - k -1) =  i. 
\]
The colour of $(u_\ell, u_{s_{2n}(\ell + t)})$ is
\[
s_n\bigg(\bigg\lceil \frac{(2k-2i)+1}{2} \bigg\rceil + (i + j - k) - 1\bigg) =
s_n(k-i+1 + i + j - k -1) =  j. 
\]
As \(i < j\), we have \(s < t\). Now,
\[
t - s = (2k - 2i) - (2k - 2j + 1) = 2(j-i) -1,
\]
which is not a multiple of $2n$. Hence the colour of $(u_{s_{2n}(\ell+s)},
u_{s_{2n}(\ell+t)})$ is 
\begin{align*}
s_n\bigg(\bigg\lceil \frac{(2(j-i)-1)+1}{2} \bigg\rceil + (i + j - k) + (2k -2j
  + 1) - 1\bigg) &=   \\ 
s_n(j-i +i + j -k + 2k  -2j + 1  -1) &= k.
\end{align*}
So, $(u_{\ell},u_{s_{2n}(\ell+s)},u_{s_{2n}(\ell+ t)})$ is a triangle with edges
coloured by \(i\), \(j\), and \(k\).
Next, assume $k\neq j$ and put
\[
\ell \deq s_{2n}(i + j - k), \quad s \deq 2k - 2j, \quad \textrm{and} \quad
t \deq 2k - 2i +1.
\]
Again, $s,t\in \mathbb{Z}\setminus\{2mn: m\in\mathbb{Z}\}$.
By calculations similar to the first case, we get that the colour of
$(u_\ell, u_{s_{2n}(\ell + s)})$ is $i$, and the colour of
$(u_\ell, u_{s_{2n}(\ell + t)})$ is $j$. Moreover,
\[
t - s = (2k - 2i +1) - (2k - 2j) = 2(j-i) +1,
\]
so the colour of $(u_{s_{2n}(\ell+s)}, u_{s_{2n}(\ell+t)})$ is
\begin{align*}
  s_n\bigg(\bigg\lceil \frac{(2(j-i)+1)+1}{2} \bigg\rceil +
  (i + j  - k) + (2k -2j) - 1\bigg) &=   \\
s_n(j-i+1 + i +j -k +2k -2j -1) &= k.
\end{align*}

Choosing between the two cases, as needed for
(i) \(k = i, k\neq j\), (ii) \(k \neq i, k = j\), or (iii) \(k\neq i,j\) we
can construct all non-monochromatic triangles. 
\end{proof}

\begin{theorem}
For any natural number $n$ the algebra $\mathfrak{E}_{n+1}^{\{2,3\}}$ is qualitatively
representable on the graph $K_{2n} $.
\end{theorem}

\begin{proof}
Immediate by Lemma~\ref{map-s}.
\end{proof}

\section{Algebras $\mathfrak{E}_{n+1}^{\{1,3\}}$} 

These algebras are better known as \emph{Lyndon algebras}.  Lyndon \cite{LR61} showed
that for $n\geq 4$, strong representations of the algebra $\mathfrak{E}_{n+1}^{\{1,3\}}$ precisely correspond to affine planes of order $n-1$, so that by the Bruck-Ryser Theorem \cite{BR49} and standard field constructions, there are infinitely many $n$ for which $\mathfrak{E}_{n+1}^{\{1,3\}}$ is strongly representable, and infinitely many for which it is not strongly representable.  This was used by Monk to show that the strongly representable relation algebras admit no finite axiomatisation \cite{MJ64}.  We now explore how qualitative and even feeble representations of $\mathfrak{E}_{n+1}^{\{1,3\}}$ relate similarly to geometric objects.  The natural starting point will be feeble representations, which we find corresponds to parallelisms of linear spaces, in the sense of finite geometry.  We direct the reader to a text on linear spaces such as \cite{batbeu} for further background on linear spaces, though the connection with feeble representations of Lyndon algebras is presented here for the first time.

A \emph{linear space} is a system $(G,\mathscr{L})$, where $G$ is a set of points, and $\mathscr{L}$ a
family of subsets of $G$ called \emph{lines} satisfying the following axioms:
\begin{enumerate}
\item[(LS1)] Every pair of distinct points is contained in a line;
\item[(LS2)] Every pair of distinct lines have intersection that is either empty or a singleton;
\item[(LS3)] Every line contains at least two points.
\end{enumerate}
An obvious consequence of LS1 and LS2 is that each pair of points $p$, $q$ are
contained within a unique line, which  we denote by $\overline{pq}$.  As usual,
points lying on the same line are said to be \emph{collinear}, two lines are
\emph{incident} if they intersect to a common point. 

A \emph{parallelism} of a linear space $(G,\mathscr{L})$ is an equivalence relation $\varpi$ on $\mathscr{L}$ with the property that lines within a common block do not intersect.  The blocks of~$\varpi$ are referred to as \emph{parallel classes} and lines within a block are \emph{parallel}.  The identity relation on $\mathscr{L}$ is always an example of a parallelism, and we refer to this as the \emph{trivial parallelism}.  But in general there may be many other parallelisms, with various numbers of parallel classes.  The next theorem shows that $n$-block parallelisms of linear spaces precisely capture feeble representations of the algebra $\mathfrak{E}_{n+1}^{\{1,3\}}$, modulo the choice of bijection between diversity atoms and parallel classes.
\begin{theorem}\label{thm:feebleLyndon}
Every linear space $(G,\mathscr{L})$ with $n$-block parallelism $\varpi$ yields a feeble representation of $\mathfrak{E}_{n+1}^{\{1,3\}}$ by relating $p,q\in G$ by the $i^{\rm th}$ colour whenever $p$ is collinear with $q$ by way of a line in the $i^{\rm th}$ block of $\varpi$.  Moreover, every feeble representation $\mathfrak{E}_{n+1}^{\{1,3\}}$ arises in this way.
\end{theorem}
\begin{proof}
For the forward direction, we consider the $n$-block parallelism $\varpi$ of $(G,\mathscr{L})$ and the defined mapping of diversity atoms of $\mathfrak{E}_{n+1}^{\{1,3\}}$ to parallel classes as given.  For convenience, we allow the parallel classes to share the same name as the colour of the atom to which they are matched.  So the (attempted) feeble representation~$a^\theta$ of a diversity atom $a$ labels the edge from $p$ to $q$ if $p$ is collinear to $q$ by way of a line in $a$ (the identity atom is represented, as required, by the identity relation).  By Axioms LS1 and LS2, all pairs receive a unique colour.  By LS3, all colours appear, so that $\theta$ is a feeble representation of $\mathfrak{E}_{n+1}^{\{1,3\}}$.

For the converse direction, let $\theta$ be a feeble representation of $\mathfrak{E}_{n+1}^{\{1,3\}}$ on a set $G$.  To define the lines $\mathscr{L}$, for each diversity atom $a$ we will include maximal cliques with respect to edges $a^\theta$ as lines in $\mathscr{L}$.  To define the parallelism $\varpi$, we define lines to be parallel if they are cliques with respect to a common atom.  Because each pair of points in $G$ are related by $a^\theta$ for some atom $a$, we have that Axiom LS1 holds.  Because no edge is coloured by more than one atom, we have that Axiom LS2 holds.  It is trivial that Axiom LS3 holds.  Because $\theta$ is a feeble representation, each atom labels at least one edge, so that the number of  blocks in $\varpi$ is exactly~$n$.  Because there are no dichromatic triangles in the feeble representation $\theta$ we have that the lines within common blocks of $\varpi$ do not intersect, so that $\varpi$ is a parallelism of $(G,\mathscr{L})$.  Finally, it is clear that the representation of $\mathfrak{E}_{n+1}^{\{1,3\}}$ over $(G,\mathscr{L})$, $\varpi$ agrees with the one defined in the first half of the proof, with the obvious matching of atoms to parallel classes.
\end{proof}

We mention in passing that an alternative presentation here is to begin with a linear space $(G,\mathscr{L})$ in which there is a distinguished line $L_\infty$ (the line at infinity) that is incident with all other lines and contains precisely $n$ points.  Provided that every line other than $L_\infty$ contains at least $3$ points, we may obtain a linear space $(G\backslash L_\infty,\mathscr{L}\backslash\{L_\infty\})$ and define an $n$-block parallelism by letting two lines be parallel if they were incident in $(G,\mathscr{L})$ to a common point on $L_\infty$.  Conversely, any parallelism $\varpi$ of a linear space gives rise to a new linear space by letting the blocks of $\varpi$ be treated as additional new points, and the set of blocks of $\varpi$ be considered as an additional new line (``at infinity'').  This of course is a general case of the familiar interplay between projective planes and affine planes.

Since $3\in \{1,3\}$, applying Observation~\ref{obs:3-feeble} immediately gives
the next result.
\begin{obs}
Feeble representations of Lyndon algebras, and hence the corresponding linear spaces,
exist for all $n\geq 3$.
\end{obs}
A simple application of Theorem \ref{thm:feebleLyndon} yields another solution.
The \emph{near pencil} is a linear space on $n\geq 3$ points
$p_0,p_1,\dots,p_{n-1}$, where the lines are
$\{p_1,\dots,p_{n-1}\}$ along with $\{p_0,p_i\}$ for each $i=1,\dots,n-1$.  This near pencil has $n$ lines, and with the trivial parallelism yields a further easy feeble representation of $\mathfrak{E}_{n+1}^{\{1,3\}}$.

We now consider qualitative representations; each triangle type has a natural geometric interpretation, the first that there are large enough lines, the second that there are enough points in general position.
\begin{theorem}\label{thm:qualLyndon}
A linear space $(G,\mathscr{L})$ with $n$-block parallelism $\varpi$ represents all monochromatic triangles in $\mathfrak{E}_{n+1}^{\{1,3\}}$ precisely when \textup{(LS4)} holds and all trichromatic triangles precisely when \textup{(LS5)} holds\textup:
\begin{enumerate}
\item[LS4] each parallel class contains a line with at least $3$ points.
\item[LS5] each triple of parallel classes are witnessed by $3$ points in general position.  That is, for parallel classes $d_1,d_2,d_3$ there are points $p_1,p_2,p_3$ in general position, with $\overline{p_1p_2}\in d_1$, $\overline{p_2p_3}\in d_1$, $\overline{p_3p_1}\in d_3$.
\end{enumerate}
\end{theorem}
\begin{proof}
This is immediate, as it simply states the condition on representations in the geometric formulation shown equivalent by Theorem \ref{thm:feebleLyndon}.  
\end{proof}
We will refer to a linear space with parallelism $(G,\mathscr{L},\varpi)$ as an
\emph{affine Lyndon geometry} if it satisfies the extra axioms LS4 and LS5.  The
linear space obtained from an affine Lyndon geometry by adjoining the line at
infinity and associated directions, will be called a \emph{Lyndon geometry};
note that a Lyndon geometry is a linear space with a distinguished line.  As
with affine planes and projective  planes, these two concepts in some sense
differ only in that parallel classes and the parallelism are given the name
``points at infinity'' and ``line at infinity''.  In order to match the case for affine planes, we let the \emph{order} of affine Lyndon geometry $(G,\mathscr{L},\varpi)$ be the number $n$ such that $\varpi$ has $n+1$ blocks. 

We now show that affine Lyndon geometries of every order exist.



\begin{theorem}\label{thm:drop}
Let $(G,\mathscr{L})$ be an affine plane of order $n\geq 3$ and consider any set $D$
of $k\leq n-2$ points in $G$.  
Define a parallelism $\varpi$ on the subspace on
$G\setminus D$ by defining two lines parallel if\textup{:} in $(G,\mathscr{L})$
they pass through a unique point of $D$\textup{;} and otherwise if they were
parallel in $(G,\mathscr{L})$.  The parallelism $\varpi$ has $n+k+1$ blocks. 
\end{theorem} 
\begin{proof}
Let $\mathscr{L}_D$ denote the lines in the defined subspace: so $K\in
\mathscr{L}_D$ is of the form $L\setminus D$ for a unique $L\in \mathscr{L}$. 
For $L\in\mathscr{L}$ we will write $L^-$ to denote $L\setminus D$.
We will refer to blocks of $\varpi$ corresponding to a parallel class in
$(G,\mathscr{L})$, as \emph{old} directions, and those consisting of the pencil of
lines that pass through precisely one single point $p\in D$ as \emph{new}
directions. Since each line in $\mathscr{L}$ contains $n$ points, each
line in $\mathscr{L}_D$ contains at least $2$ points, verifying the axiom (LS3). 
The axioms (LS1) and (LS2) are inherited from $(G,\mathscr{L})$,
so $(G\setminus D,\mathscr{L}_D)$ is a linear space. 

To verify axiom (LS4) note that as 
there are $n$ lines in each direction of the affine plane
$(G,\mathscr{L})$, and at most $n-2$ of these lines are through unique points in
$D$, there are at least $2$ lines in each old direction. Let $m$ be the
number of lines in an old direction $o$. These lines contain  
at least $mn-(n-2)$ points, and as $n\geq 3$ and $m\geq 2$
we have $mn-n+2 > 2m$, showing that at least one of the $m$ lines must
contain strictly more than 2 points. 
Analogously, there
are $n+1$ lines through each point $p$ of $D$, and at most $k-1\leq n-3$ of them
can pass through a second point of $D$, so each new direction contains at least
$4$ lines. As above, it follows that at least one
of these lines contains at least $3$ points. It is now also clear that the
parallelism $\varpi$ has $n+k+1$ blocks.

For axiom (LS5), consider three
directions from $\varpi$.  If all three are old, say, $o_1$, $o_2$ and $o_3$,
then as there are $n$ lines of $\mathscr{L}$ in $o_1$,
there is a line $L_1\in \mathscr{L}\cap o_1$, with $L_1\cap D=\varnothing$,
so $L_1\in\mathscr{L}_D\cap o_1$, that is, $L_1^- = L_1$. Similarly, there is a line
$L_2\in \mathscr{L}\cap o_2$ with $L_2^- = L_2$. 
Let $p$ be the unique point in $L_1\cap L_2$. As there are $n$ lines
in $\mathscr{L}\cap o_3$, and $|D\cup\{p\}|= k+1 \leq n-1$, there is
a line $L_3\in \mathscr{L}\cap o_3$ with
$L_3\cap (D\cup\{p\})=\varnothing$, so that $L_3^- = L_3$.
Since $L_1$, $L_2$ and $L_3$
are lines from $\mathscr{L}$, not all through the same point, they intersect pairwise at
3 distinct points. Hence, $p\in L_1\cap L_2$, $q\in L_1\cap L_3$ and $r\in
L_2\cap L_3$ are 3 points in general position in $G\setminus D$.
But since $L_1=L_1^-$, $L_2=L_2^-$ and $L_3=L_3^-$ they witness (LS5)
for $o_1$, $o_2$ and $o_3$ in $(G\setminus D, \mathscr{L}_D)$.

Next, let $o_1$ and $o_2$ be old directions and let $d_3$ be a new direction.
Reasoning as before, we find lines $L_1\in\mathscr{L}\cap o_1$
and $L_2\in\mathscr{L}\cap o_2$ such that $L_1^- = L_1$, $L_2^- = L_2$, 
and $L_1\cap L_2 = \{p\}$ with $p\in G\setminus D$.  
As there are $n+1$ lines in $\mathscr{L}$ passing through~$d_3$,
there is a line $L_3\in \mathscr{L}$ passing through~$d_3$, with
$L_3\cap ((D\cup\{p\})\setminus\{d_3\})=\varnothing$.
Then, $p\in L_1\cap L_2$, $q\in L_1\cap L_3$ and $r\in L_2\cap L_3$ are 
3 points in general position in $G\setminus D$.
Moreover, $L_3^- = L_3\setminus\{d_3\}$, so
$L_1^-$, $L_2^-$ and $L_3^-$ witness (LS5) for $o_1$, $o_2$ and $d_3$
in $(G\setminus D, \mathscr{L}_D)$.

Further, let $o_1$ be an old direction and let $d_2$ and $d_3$ be new
directions. As there are $n+1$ lines in $\mathscr{L}$ passing through $d_2$,
there are $n$ lines in $\mathscr{L}\setminus o_1$
passing through~$d_2$, so we find a line $L_2\in \mathscr{L}\setminus o_1$ such that 
$d_2\in L_2$ and $L_2\cap (D\setminus\{d_2\}) =\varnothing$.
Let $o_2$ be the direction of $L_2$ in $(G,\mathscr{L})$.
As there there are $n-1$ lines in $\mathscr{L}\setminus (o_1\cup o_2)$
passing through $d_3$, we find $L_3\in\mathscr{L}\setminus (o_1\cup o_2)$
with $d_3\in L_3$ such that $L_3\cap (D\setminus\{d_3,\}) =\varnothing$.
Then $L_2$ and $L_3$ intersect at a point $r\in G\setminus D$. Now
consider $o_1$. There are $n$ lines of $\mathscr{L}$ in $o_1$,
at most $k+1 = n-1$ of them having nonempty intersections with
$D\cup \{r\}$, so there is a line $L_1\in o_1$ such that
$L_1\cap(D\cup \{r\}) = \varnothing$. Then, $p\in L_1\cap L_2$, $q\in L_1\cap
L_3$ and $r\in L_2\cap L_3$ are 3 points in general position in $G\setminus D$.
So the lines
$L_1^- = L_1$, $L_2^- = L_2\setminus \{d_2\}$ and $L_3^- = L_3\setminus\{d_3\}$,
witness (LS5) for $o_1$, $d_2$ and $d_3$ in $(G\setminus D, \mathscr{L}_D)$. 

Finally, let $d_1$, $d_2$ and $d_3$ be new directions. We find 
$L_2\in \mathscr{L}$ such that $d_2\in L_2$ and $L_2\cap
(D\setminus\{d_2\}) =\varnothing$. Let $o_2$ be the direction of 
$L_2$ in $(G,\mathscr{L})$. Then we find $L_3\in\mathscr{L}\setminus o_2$,
with $d_3\in L_3$ and $L_3\cap (D\setminus\{d_3,\}) =\varnothing$, so 
$r\in L_2\cap L_3$ is a point in $G\setminus D$. Let
$o_3$ be the direction of  $L_3$ in $(G,\mathscr{L})$. There are 
$n-1$ lines in $\mathscr{L}\setminus(o_2\cup o_3)$ passing through
$d_1$, and $|(D\setminus\{d_1\})\cup\{r\}| = k\leq n-2$, so there is a line
$L_1\in\mathscr{L}\setminus(o_2\cup o_3)$ passing through
$d_1$ with $L_1\cap ((D\setminus\{d_1\})\cup\{r\}) =\varnothing$. 
Then, $p\in L_1\cap L_2$, $q\in L_1\cap
L_3$ and $r\in L_2\cap L_3$ are 3 points in general position in $G\setminus D$,
and the lines $L_1^- = L_1\setminus\{d_1\}$, 
$L_2^- = L_2\setminus \{d_2\}$ and $L_3^- = L_3\setminus\{d_3\}$,
witness (LS5) for $d_1$, $d_2$ and $d_3$ in $(G\setminus D, \mathscr{L}_D)$.  
\end{proof}

\begin{cor}\label{cor:notstrLyndon}
There exist affine Lyndon geometries of all orders greater than $3$;
equivalently,  $\mathfrak{E}_{n+1}^{\{1,3\}}$ is qualitatively representable
for all $n\geq 3$. 
\end{cor} 
\begin{proof}
By Theorem \ref{thm:qualLyndon} it suffices to cover all $n\geq 3$ using
Theorem~\ref{thm:drop}.  There are affine planes of every prime order
$p$,  so Theorem~\ref{thm:drop} shows that there are affine Lyndon
geometries of all orders $p,p+1,\dots,2(p-1)$.  All numbers $n\geq 2$ lie
in the interval $[p, 2(p-1)]$ for some prime $p$: this is trivial to verify for
small $n$, and follows, for example, from Nagura's variant of Bertrand's
postulate~\cite{nag}, showing that for $m\geq 25$ there is a prime between $m$
and $5m/4$. 
\end{proof}

\section{Infinite cardinalities}

The definition of chromatic algebras $\mathfrak{E}_{n}^S$ easily extends to
allow for the case $n=\kappa$ for an infinite cardinal $\kappa$.
Strong representations for $\mathfrak{E}_{\kappa}^S$ are known for $S=\{1,3\}$
(by way of affine planes of order $\kappa$), while strong representations for
$\mathfrak{E}_{\kappa}^S$ in the case of $S=\{1,2\},\{2,3\},\{1,2,3\}$  are
easily achievable using the game-theoretic methods of
Hirsch and Hodkinson~\cite{HH02}, but they are somewhat
convoluted, being subdirect products of representations of
countable subalgebras of $\mathfrak{E}_{\kappa}^S$ (even in the case
$\kappa=\omega$). We will not investigate them here.
The remaining cases do not admit strong  representations: for $\varnothing$ and
$\{1\}$ this is trivial, while for $S=\{3\}$ and $S=\{2\}$, the situation is
discussed in the relevant Sections~\ref{sec:quasigr} and~\ref{sec:2}.   

For qualitative representations there emerges an interesting  contrast with the
finite colour results, namely, $\mathfrak{E}_{\kappa}^{\{2\}}$ becomes qualitatively
representable for all infinite~$\kappa$, despite  $\mathfrak{E}_{n+1}^{\{2\}}$
not being qualitatively representable for any finite $n$.  Moreover,
there is a direct, uniform method of constructing representations for all
infinite cases.

\begin{theorem}\label{thm:uncountable}
Let $\kappa$ be an infinite cardinal.  
For $S\subseteq \{1,2,3\}$, the algebra $\mathfrak{E}_{\kappa}^S$ has a
qualitative representation if and only if $S\cap \{2,3\}\neq \varnothing$. 
\end{theorem}
\begin{proof}
For $S=\varnothing$ or $S=\{1\}$ it is trivial that no representations exist.  

Now assume that $3\in S$, and let $C_\kappa$ be the set of all
colours. There are $\kappa$ triangles $(T_\alpha: \alpha<\kappa)$ to be
witnessed in order to achieve a qualitative representation of
$\mathfrak{E}_{\kappa}^S$. We will construct a chain of complete networks
$(N_\alpha: \alpha<\kappa)$ and a chain of sets of colours 
$(C_\alpha: \alpha<\kappa)$ such that $C_\alpha$ is the set of all colours used
in $N_\alpha$. We use $\otimes$ and $\oplus$ for cardinal multiplication and
addition, and $+$ for ordinal addition.
Let $N_0$ be the empty network. Assume inductively that
a chain of complete networks
$(N_\beta: \beta<\alpha)$ has been constructed, such that
(i) each $N_\beta$ witnesses the triangles $T_\gamma$ for all $\gamma< \beta$,
(ii) $|N_\beta| \leq 3\otimes |\beta|$, and
(iii) $|C_\beta| \leq 9\otimes|\beta|\otimes |\beta|$.
Construct $N_\alpha$ as follows.
\begin{itemize}
\item If $\alpha = \gamma+1$ for some $\gamma$, let
$N'_\alpha := N_\gamma\uplus T_\gamma$. Then
$|N'_\alpha| = |N_\gamma|\oplus 3 \leq 3\otimes |\gamma|\oplus 3 \leq
3\otimes |\alpha|$,  and 
$N'_\alpha$ witnesses all triangles $T_\beta$ for $\beta<\alpha$. So
$N'_\alpha$ satisfies (i) and (ii) but it is not complete (except for
$\gamma=0$). Let $M'_\alpha :=\{\{x,y\}: x\in N_\gamma,\ y\in T_\gamma \}$.
Then, $|M'_\alpha| = 3\otimes|N_\gamma|$, and putting $C'_\alpha$ for the set of
colours used in $N'_\alpha$ we get
$|C'_\alpha| = 
|C_\gamma|\oplus 3\leq 9\otimes|\gamma|\otimes |\gamma| \oplus 3<\kappa$. 
Hence $|C_\kappa\setminus C_\alpha'| = \kappa$ and therefore 
we can choose from $C_\kappa\setminus C_\alpha'$ a unique
colour for each $\{x,y\}\in M'_\alpha$ and thus complete $N_\alpha'$ to
$N_\alpha$. Then $|N_\alpha| = |N'_\alpha| \leq 3\otimes |\alpha|$,  and
$|C_\alpha| \leq |C_\alpha'\cup M'_\alpha| \leq
9\otimes|\gamma|\otimes |\gamma| \oplus 3\oplus 3\otimes|N_\gamma|\leq
9\otimes|\gamma|\otimes |\gamma|\oplus 3\oplus 9\otimes|\gamma| 
\leq 9\otimes |\alpha|\otimes|\alpha|$. Thus,
(i), (ii) and (iii) are satisfied.

\item If $\alpha$ is a limit ordinal, let $N_\alpha :=
\bigcup_{\beta<\alpha}N_\beta$. Then $N_\alpha$ is a complete network witnessing
the triangles $T_\beta$ for all $\beta< \alpha$, so
(i) is satisfied. Moreover,
$|N_\beta|\leq  3\otimes |\beta| \leq |\alpha|$ for each $\beta<\alpha$, so
$N_\alpha$ is a  union of
$|\alpha|$ sets, each of cardinality at most $|\alpha|$; hence
$|N_\alpha|\leq |\alpha|$, satisfying (ii). Similarly,
$|C_\beta|\leq 9\otimes|\beta|\otimes |\beta|\leq |\alpha|$ for each
$\beta<\alpha$, so 
$|C_\alpha| = |\bigcup_{\beta<\alpha}C_\beta|\leq |\alpha|$,
showing that (iii) is satisfied.
\end{itemize}
The union $\bigcup_{\alpha<\kappa} N_\alpha$ is a well-defined
qualitative representation of $\mathfrak{E}_{\kappa}^S$.   
If $3\notin S$ but $2\in S$ we can apply same methodology as for $3\in S$,
except that now to complete $N'_\alpha$ to $N_\alpha$ for a successor $\alpha$   
we select a single colour from $\kappa$ unused ones for all the missing edges. 
\end{proof}

\section{Summary}

The table below summarises the known results on representability of finite chromatic
algebras. The main question that remains open is  
for which $n$ strong representations of Ramsey algebras
$\mathfrak{E}_{n+1}^{\{2,3\}}$ exist.
\begin{center}
\begin{tabular}{c|c|c|c}
  \hline
  \multicolumn{4}{c}{Representations of $\mathfrak{E}_{n+1}^S$}\\ 
  \hline
    $S$  & comp. of diversity atoms & quali-rep. & strong rep. \\
  \hline
  \hline
  $\{1,2,3\}$ & $a_i\comp a_j =\begin{cases}
    0' & \text{ if } i\neq j\\
    1  & \text{ if } i = j
    \end{cases}$   &  Yes & Yes \\ 
  \hline
  $\{2,3\}$ & $a_i\comp a_j =\begin{cases}
    0' & \text{ if } i\neq j\\
    \neg a_i  & \text{ if } i = j
  \end{cases}$   &  Yes & \makecell{Yes for $ n\lesssim 2000$\\
                                  except 8 and 13\\
                                  Not known in general} \\
  \hline
  $\{1,3\}$ & $a_i\comp a_j =\begin{cases}
    \neg a_i\wedge\neg a_j\wedge 0' & \text{ if } i\neq j\\
    1'\vee a_i  & \text{ if } i = j
  \end{cases}$   &  \makecell{Yes iff $n>3$}   & \makecell{Yes iff\\ 
                                    $\exists$ proj. plane\\  
                                    of ord. $n-1$} \\ 
  \hline
  $\{1,2\}$ & $a_i\comp a_j =\begin{cases}
    a_i\vee a_j  & \text{ if } i\neq j\\
    1  & \text{ if } i = j
  \end{cases}$   &  Yes & \makecell{Yes\\
                                  (only infinite rep's)} \\ 
\hline
  $\{3\}$ & $a_i\comp a_j =\begin{cases}
    \neg a_i\wedge\neg a_j\wedge 0' & \text{ if } i\neq j\\
    1'  & \text{ if } i = j
  \end{cases}$   &  Yes iff $n$ is odd  & Yes iff $n = 3$ \\ 
\hline
  $\{2\}$ & $a_i\comp a_j =\begin{cases}
    a_i\vee a_j & \text{ if } i\neq j\\
    \neg a_i  & \text{ if } i = j
  \end{cases}$   &  No  & No \\ 
\hline
  $\{1\}$ & $a_i\comp a_j =\begin{cases}
    0 & \text{ if } i\neq j\\
    1'\vee a_i  & \text{ if } i = j
  \end{cases}$   &  Yes iff $n=1$  & Yes iff $n=1$ \\ 
\hline
  $\{\varnothing\}$ & $a_i\comp a_j =\begin{cases}
    0 & \text{ if } i\neq j\\
    1'  & \text{ if } i = j
  \end{cases}$   &  Yes iff $n=1$  & Yes iff $n=1$ \\ 
\hline
\end{tabular}
\end{center}

\section*{Acknowledgements}
The authors wish to thank Michael Payne for his discussions around aspects of
this paper, particularly around developments leading to Theorem~\ref{thm:drop}.

\end{document}